\documentclass[UTF-8,reqno,12pt]{amsart}
\usepackage{enumerate}
\allowdisplaybreaks[4]
\usepackage{amssymb,url,color}
\usepackage{hyperref}
\usepackage{amsmath,amsthm}
\usepackage{mathrsfs}

\newtheorem{thm}{Theorem}[section]
\theoremstyle{Condition}

\newtheorem{lem}{Lemma}[section]
\newtheorem{rem}{Remark}[section]

\theoremstyle{Problem}

\theoremstyle{Assumption}
\newtheorem{assum}{Assumption}[section]
\theoremstyle{Definition}

\numberwithin{equation}{section}

\def\beq{\begin{equation}}
\def\deq{\end{equation}}
\allowdisplaybreaks 

\bfseries\rmfamily 

\def\cD{{\mathcal D}}

\def\cF{{\mathcal F}}

\def\mE{{\mathbb E}}

\def\mI{{\mathbb I}}

\def\mM{{\mathbb M}}
\def\mN{{\mathbb N}}

\def\mP{{\mathbb P}}

\def\mR{{\mathbb R}}

\def\sB{{\mathscr B}}

\def\sF{{\mathscr F}}

\def\geq{\geqslant}
\def\leq{\leqslant}

\def\HS{{\mathrm{\tiny HS}}}

\def\a{\alpha}

\def\b{\beta}

\def\d{\delta}

\def\s{\sigma}

\def\[{{\Big[}}
\def\]{{\Big]}}
\def\<{{\langle}}
\def\>{{\rangle}}
\def\({{\Big(}}
\def\){{\Big)}}

\def\dif{{\rm d}}

\def\no{\nonumber}
\def\={&\!\!=\!\!&}
\def\bt{\begin{theorem}}
\def\et{\end{theorem}}
\def\bl{\begin{lemma}}
\def\el{\end{lemma}}
\def\br{\begin{rem}}
\def\er{\end{rem}}

\begin{document}

\title
[Euler-Maruyama scheme for SDEs]
{Euler-Maruyama scheme for SDEs with Dini continuous coefficients}

\author{Zhen Wang, Yu Miao, Ren Jie}

\address{Zhen Wang: College of Mathematics and Information Statistics, Henan Normal University, Xinxiang, Henan, 453000, P.R.China}
\email{wangzhen881025@163.com}
\address{Yu Miao: College of Mathematics and Information Statistics, Henan Normal University, Xinxiang, Henan, 453000, P.R.China}
\email{yumiao728@gmail.com}
\address{Ren Jie: College of Mathematics and Information Statistics, Henan University of Economics and Law, Zhengzhou, Henan, 450000, P.R.China}
\email{20130006@huel.edu.cn}

\thanks{This work is partially supported by the NNSFC grants of China (No. 11971154) and (No.11901154) }

\begin{abstract}
In this paper, we show the convergence rate of Euler-Maruyama scheme for non-degenerate SDEs with Dini continuous coefficients, by the aid of the regularity of the solution to the associated Kolmogorov equation. We obtain the same conclusions using a simple and clever way to simplify the proof and weaken the conditions in \cite{BHY} by the properties of Dini continuous and Taylor expansion.

\end{abstract}

\subjclass[2000]{}

\maketitle

\section{Introduction}
Let fix $T>0$. Considering the following stochastic differential equation in $\mR^d$:
\beq\label{1}
X_t=X_0+\int_0^tb(s,X_s)\dif s+\int_0^t\s(s,X_s)\dif W_s,\ \ X_0\in\mR^d,
\deq
where $b:[0,T]\times\mathbb{R}^d\rightarrow \mathbb{R}^d$
and $\sigma:[0,T]\times\mathbb{R}^d\rightarrow
\mathbb{R}^d\otimes\mathbb{R}^d$ are two Borel measurable functions, $\{W_t,t\in[0,T]\}$ is an $d$-dimensional
standard Brown motion defined on a complete filtered probability space $(\Omega, \mathscr{F}, \mP; (\sF_t)_{t\geq 0})$,
and the initial value $X_0$ is $\cF_0$-measurable $\mathbb{R}^d$-valued random variable.

The Euler-Maruyama scheme of (\ref{1}) is
\beq\label{2}
 Y_t=Y_0+\int_0^tb(\eta_{\delta}(s),Y_{\eta_{\delta}(s)})\dif s+\int_0^t\s(\eta_{\delta}(s),Y_{\eta_{\delta}(s)})\dif W_s,\ \ X_0=Y_0,
\deq
where  $\delta:=\frac{T}{N}\in(0,1)$, $\eta_{\delta}(s):=\left[\frac{s}{\delta}\right]\delta$, $s\in\left[\left[\frac{s}{\delta}\right]\delta,(\left[\frac{s}{\delta}\right]+1)\delta\right)$, for the sufficiently large integer $N\in\mN$. And the discretized  scheme of (\ref{2})
is analytically tractable on computer application of engineering, physical, finance, biology, etc..

If the coefficient of SDE is Lipschitz continuous, there are many previous research results.
If $\sigma$ is an identity matrix and the coefficient $b$ is Lipschitz continuous in space and $\frac{1}{2}$-H\"older continuous in time then for any $p>0$, there exists $C_p>0$ such that the  Euler-Maruyama scheme is the strong rate of $\frac{1}{2}$ (see for example \cite{KP}). Yan \cite{Y} proved the rate of convergence in $L^1$-norm sense for a range of SDEs, where the drift coefficient is Lipschitz and the diffusion coefficient is H\"older continuous, by means of the Meyer Tanaka formula. \cite{GR} extended \cite{Y} to the convergence rate in $L^p$-norm, by the Yamada-Watanabe approximation. Our finding partly improves upon recent results in \cite{MT} \cite{NT} and \cite{GR}, as well as the well-known ones in \cite{HK} \cite{G}.

However ,in many applications, the coefficients $b$ and $\s$ are not Lipschitz continuous.
Zhang \cite{Z} is proved Euler-Maruyama approximation for SDE to converge uniformly to the solution in $L^p$-space with respect to the time and starting points under non-Lipschitz coefficients. If the drift coefficient is the Dini-H\"older continuous, Gy\"ongy and R\'asonyi \cite{GR} implied the order of strong rate of convergence for one dimensional SDEs. And the case of $d$-dimension is introduced in \cite{NT}, using a Yamada-Watanabe approximation technique (see \cite{YW}). Ngo and Taguchi \cite{NT1} showed that the rate under the diffusion coefficient is bounded variation and  H\"older continuous. In \cite{PT}, the diffusion coefficient $\sigma$ is  an identity matrix, the drift coefficient $b$ is bounded $\beta$-H\"older continuous with $\beta\in(0,1)$ in space and $\eta$ -H\"older continuous in time with $\eta\in[1/2,1]$, then for any $p\geq 1$, the strong rate of convergence can be obtained.
Bao, Huang and Yuan \cite{BHY} discussed the strong convergence rate of Euler-Maruyama for non-degenerate SDE with rough coefficients, where the drift term is Dini-continuous and unbounded, by the regularity of non-degenerate Kolmogrov equation. 

In this paper, we first study the convergence rate of the Euler-Maruyama scheme of (\ref{2}), where the drift term $b$ and the diffusion term $\s$ are the uniformly bounded, $b$ and $\s$ satisfy correlated conditions of Dini-continuous (see Assumption \ref{assum-1}), which is weaken the the conditions, simply the proof and obtains the same results in \cite{BHY}. In addition, we also prove convergence rate for the non-degenerate SDEs with unbounded coefficients, which method is mainly based on the regularity of the solution to the Kolmogorov equation associated to the SDE (\ref{1}).

This paper is structured as follows. In the next section, we introduce some notations and the main results. All proofs are deferred to Section 3.
\section{Main results}
\subsection{Notations}
~\\

In this section, we recall the foundational definition and notations involved in the paper.

     Let $\sB(\mR^d)$ be the Borel-$\s$-algebra on $\mR^d$. Set $\nabla:=D=(\frac{\partial}{\partial x_1},\cdots,\frac{\partial}{\partial x_d})^*$, $D^2=(\frac{\partial^2}{\partial x_i\partial x_j})_{1 \leq i,j\leq d}$ and $\Delta=\sum_{i=1}^d\frac{\partial^2}{\partial x_i^2}$, where $*$ is the transpose of a vector or matrix. Take $\|\cdot\|$ and $\|\cdot\|_{\HS}$ stand for the usual operator norm and the Hilbert-Schmidt norm, respectively.

     Meanwhile, we introduce some space of function:
~\\
$\bullet$ $\|f\|_{T,\infty}=\sup_{t\in[0,T],x\in\mR^d}\|f(t,x)\|$, where an operator-valued map $f$ is on $[0,T]\times\mR^d$.
~\\
$\bullet$ $\mM_{\mathrm{non}}^d$ denotes the collection of all nonsingular $d\times d$-matrices.
~\\
$\bullet$ $C_b^{\beta}(\mR^d,\mR^k), \beta\in(0,1)$ denotes the set of all function from $\mR^d$ to $\mR^k$ which are bounded and $\beta$-H\"older continuous functions. Hence if $f\in C_b^{\beta}(\mR^d,\mR^k)$, then
$$
\sup_{x,y\in\mR^d,x\neq y}\frac{|f(x)-f(y)|}{|x-y|^{\beta}}<\infty.
$$
~\\
$\bullet$ For $a<b$, we write $C_b^{\beta}([a,b])$ for $C([a,b];C_b^{\beta}(\mR^d,\mR^d))$ and define the norm $\|\cdot\|_{C_b^{\beta}([a,b])}$ on $C_b^{\beta}([a,b])$ by
$$
\|f\|_{C_b^{\beta}([a,b])}:=\sup_{t\in[a,b],x\in\mR^d}|f(t,x)|+\sup_{t\in[a,b],x\neq y}\frac{|f(t,x)-f(t,y)|}{|x-y|^{\beta}}.
$$

Throughout the paper, we denote the constant as $C$, the shorthand notation $a \preceq b$ stands for $a\leq Cb$. And $C$ represents a positive constant although its value may change from one appearance to the next.
\subsection{Main results}
~\\

In this paper, we study the convergence rate of Euler-Maruyama scheme,
under the following non-Lipschitz condition. In this section, we state the related assumptions and  main theorems of this paper.

Let $\cD_0$ be the family of Dini function, i.e.,
$$
\cD_0:=\left\{\phi\Big|\phi:\mR_+\rightarrow\mR_+ \ \ \mathrm{is}\ \ \mathrm{increasing}\ \ \mathrm{and}\ \  \int_0^1\frac{\phi(s)}{s}\dif s<\infty\right\}.
$$
A function $f:\mR^d\rightarrow\mR^d$ is called Dini-continuity if there exists $\phi\in\cD_0$ such that $|f(x)-f(y)|\leq \phi(|x-y|)$ for any $x,y\in\mR^d$. It is well known that every Dini-continuous function is continuous and every Lipschitz continuous function is Dini-continuous. Moreover, if $f$ is H\"older continuous, then $f$ is Dini-continuous, but not vice versa.
And set
$$
\cD:=\left\{\phi\in\cD_0|\phi^2 \ \  \mathrm{is}\ \ \mathrm{concave} \right\},
$$
for instance, a function $f$ is H\"older-Dini continuous of order $\a\in(0,1)$.

The non-Lipschitz assumptions is following:

\begin{assum}\label{assum-1}
\begin{enumerate}[(a)]
~
\item for every $t\in[0,T]$ and $x\in\mR^n$, $\s(t,x)\in\mM_{\mathrm{non}}^n$, and
$$
\|b\|_{T,\infty}+\|\s\|_{T,\infty}<+\infty,
$$
where $\|\s\|_{T,\infty}:=\sup_{0\leq t\leq T}\|\s(t,x)\|_{\mathrm{HS}}$.
\item For any $t\in[0,T]$, $\b\in(0,1)$ and $x,y\in\,\mR^d$, there exists $\phi\in\cD$ such that\\
 (regularity of $b$ and $\s$ w.r.t. spatial variables)
$$
|b(t,x)-b(t,y)|\leq |x-y|^{\b}\phi(|x-y|),\ \  \|\s(t,x)-\s(t,y)\|_{\HS}\leq \phi(|x-y|).
$$
\item For any $s, t\in [0,T]$ and $x\in\,\mR^d$, there exists $\phi\in\cD$ such that\\
(regularity of $b$ and $\s$ w.r.t. time variables)
$$
|b(s,x)-b(t,x)|+\|\s(s,x)-\s(t,x)\|_{\HS}\leq \phi(|s-t|).
$$
\end{enumerate}
\end{assum}

The following main results are stated including the convergence rate of SDEs.

\begin{thm}\label{thm1}
Suppose that Assumption \ref{assum-1} holds. For $p\geq1$ and $\b\in(0,1)$, there exists the constant $C>1$ depending on $T, p, d, \|\s\|_{T,\infty}, M,\b$, then
$$
\mE\left(\sup_{0\leq t\leq T}\left|X_t-Y_t\right|^p\right)\leq C\delta^{p/2}.
$$
\end{thm}

\begin{rem}In \cite{BHY}, the diffusion term is the rough coefficient which refers to second order continuous differentiable. However, we check the results with the regularity of the solution to the Kolmogorov equation associated to the SDE (\ref{1}), by the properties of Dini continuous and Taylor expansion instead of second order continuous differentiable, which is a simple and clever way to simplify the proof and weaken the conditions in \cite{BHY}.
\end{rem}
\begin{rem}In the article \cite{BHY}, the convergence rate is verified from the perspective of norm, while we not only get similar results but also do better conclusions using the properties of dimension. At the same time, we can also do the degenerate result with the help of the Hamiltonian system. Because the method is similar, we will not elaborate here.
\end{rem}
It seems to be a little bit stringent that  the coefficients are uniformly bounded, and the drift $b$ is global Dini-continuous,  in Theorem \ref{thm1}. Therefore, the above conditions can be weakened by the means of  uniform boundedness instead of local  boundedness and  global Dini-continous instead of local Dini-continuous, respectively.
\begin{thm}\label{thm2}Assume that for any $s, t\in[0,T]$, $\b\in(0,1)$ and for every $x\in\mR^d$ and $\s(t,x)\in\mM_{\mathrm{non}}^d$, there exists the constant $C_T$, $b$ and $\sigma$ are Borel measurable functions such that
$$
|b(t,x)|+\|\sigma(t,x)\|_{\HS}\leq C_T(1+x),\ \  x\in\mR^d,
$$
And if $b$  and  $\sigma$ satisfy
$$
|b(t,x)-b(t,y)|\leq |x-y|^{\b}\phi_k(|x-y|),\ \ |x|\vee|y|\leq k,
$$
$$
\|\sigma(t,x)-\sigma(t,y)\|_{\HS}\leq \phi_k(|x-y|),\ \ |x|\vee|y|\leq k,
$$
$$
|b(s,x)-b(t,x)|+\|\sigma(s,x)-\sigma(t,x)\|_{\HS}\leq \phi_k(|s-t|),\ \ |x|\leq k,
$$
where $\phi_k\in\cD$.
Then for all $p\geq 1$ and $\mE|X_0|^p<\infty$, it holds that
$$
\lim_{\delta\rightarrow 0}\mE\left(\sup_{0\leq t\leq T}\left|X_t-Y_t\right|^p\right)=0
$$
\end{thm}
\begin{rem}We verify this conclusion by a method similar to Theorem 1.2 in \cite{BHY}. In \cite{BHY}, the diffusion term is the uniformly bounded and second order continuous differentiable. However, in this paper, $\s$ is the uniformly bounded, which is weaken the conditions of Theorem 1.2 in \cite{BHY} and obtains the same conclusions.
\end{rem}

\section{Proofs of main results}

We also need the following lemma for the proof.
\begin{lem}\label{lem1}
Let the coefficients $b,\sigma$ is the uniformly bounded. For $p\geq 1$ and  $t\in [0,T]$, there exists a positive constant $C>0$ depending on $T, M, p, d$, it holds that
$$
\mE\left[\phi\left(\left|Y_t-Y_{\eta_{\delta}(t)}\right|\right)^p\right]\leq C\delta^{p/2}.
$$
\end{lem}
\begin{proof}
Owing to $\phi\in\cD$, based on Taylor expansion and the properties of Dini function,  we have $\phi(0)=0$, $\phi'>0$ and $\phi''<0$, so that
$$
\aligned
\phi(t)=&\phi(0)+\phi'(0)t+\frac{\phi''(\theta t)}{2!}t^2,\ \  \theta\in(0,1), \ \ t\in\mR_+\\
\leq&\phi'(0)t:=Mt. \\
\endaligned
$$
Thus, for any $t\in\mR$,
$$
\phi(|t|)\leq M|t|.
$$
For $p\geq 1$, noticing that
$$
\aligned
\phi\left(\left|Y_t-Y_{\eta_{\delta}(t)}\right|\right)^p\leq
&M^p\left|Y_t-Y_{\eta_{\delta}(t)}\right|^p.\\
\endaligned
$$
Using Assumption \ref{assum-1} (a)-(b), we deduce that
$$
\aligned
\left|Y_t-Y_{\eta_{\delta}(t)}\right|^p &\preceq \left|\int^t_{\eta_{\delta}(t)}b(\eta_{\delta}(s),Y_{\eta_{\delta}(s)})\dif s\right|^p+\left| \int^t_{\eta_{\delta}(t)}\sigma(\eta_{\delta}(s),Y_{\eta_{\delta}(s)})\dif W_s\right|^P\\
&\preceq \delta^p+|W_t-W_{\eta_{\delta}(t)}|^p.\\
\endaligned
$$
Hence there exists a positive constant $C=C(T, M, p, d)$, such that, for $p\geq 1$,
$$
\mE\left[\phi\left(\left|Y_t-Y_{\eta_{\delta}(t)}\right|\right)^p\right]
\leq M\mE\left[\left|Y_t-Y_{\eta_{\delta}(t)}\right|^p\right]
\leq C\delta^{p/2}.
$$
\end{proof}

The following lemma is taken from Theorem 2.8  in \cite{F}, which provides the regularity of solution to the Kolmogorov equation associated to the SDE (\ref{1}).

\begin{lem}\label{cor-1}Let $T>0$, for any $\varepsilon\in(0,1)$, there exists $m\in\mN$ such that $0=t_0< t_1< \cdots<t_m= T$, for any $\varphi\in C([t_{j-1},t_{j}];C_b^{\beta}(\mR^d,\mR^d)), j=1,\cdots,m$, there is at least one solution $u$ to the Backward Kolmogorov equation
$$
\frac{\partial u}{\partial t}+\nabla u\cdot b +\frac{1}{2}\Delta u\cdot\s^2=-\varphi, \ \ on\ \  [t_{j-1},t_{j}]\times \mR^d, \ \ u(t_{j},x)=0
$$
of class
$$
u\in C\left([t_{j-1},t_{j}];C_b^{2,\beta'}\left(\mR^d,\mR^d\right)\right)\cap C^1\left([t_{j-1},t_{j}],C_b^{\beta'}\left(\mR^d;\mR^d\right)\right).
$$
For some constant $K$ depending on $j$ and for all $\beta'\in(0,\beta)$, we have
$$
\|D^2u\|_{C_b^{\beta'}\left([t_{j-1},t_{j}]\right)}\leq K\|\varphi\|_{C_b^{\beta}\left([t_{j-1},t_{j}]\right)}
$$
and for some  constant $C_0$, it holds that
$$
\|\nabla u\|_{C_b^{\beta}([t_{j-1},t_{j}])}\leq C_0(t_{j}-t_{j-1})^{1/2}\|\varphi\|_{C_b^{\beta}([t_{j-1},t_{j}])}
$$
At same time,  we can obtain
$$
\|\varphi\|_{C_b^{\beta}([0,T])}C_0(t_{j}-t_{j-1})^{1/2}\leq \varepsilon.
$$
\end{lem}

Now we can give

\begin{proof}[Proof of Theorem \ref{thm1}] Let $T>0$, for any $\varepsilon\in(0,1)$, there is $m\in\mN$,  such that $0=T_0<T_1<\cdots<T_m=T$. For $i=1, \cdots, d$ and $j=1,\cdots,m$, using Lemma \ref{cor-1},  we can get
\begin{align}\label{b}
\frac{\partial u}{\partial t}+\nabla u\cdot b +\frac{1}{2}\Delta u\cdot\s^2=-b, \ \ on\ \  [T_{j-1}, T_{j}]\times \mR^d, \ \ u(T_{j},x)=0,
\end{align}
and $u$ satisfies
\begin{align}\label{a}
\|\nabla u\|_{C_b^{\beta}([T_{j-1}, T_{j}])}\leq C_0(T_{j}-T_{j-1})^{1/2}\|b\|_{C_b^{\beta}([T_{j-1},T_{j}])}\leq \varepsilon.
\end{align}
For $ t\in[T_{j-1},T_j]$, by It\^o's formula and (\ref{b}), we have
$$
\aligned
&u(t,X_t)\\=&u(T_{j-1},X_{j-1})+\int_{T_{j-1}}^t\frac{\partial u}{\partial t}(s,X_s)\dif s+\int_{T_{j-1}}^t\nabla u(s,X_s)\dif X_s\\&+\frac{1}{2}\int_{T_{j-1}}^t\Delta u(s,X_s)\dif \langle X_s,X_s \rangle\\
=&u(T_{j-1},X_{j-1})-\int_{T_{j-1}}^tb(s,X_s)\dif s+\int_{T_{j-1}}^t\<\nabla u(s,X_s),\s(s,X_s)\>\dif W_s.\\
\endaligned
$$
Similarly, we have
$$
\aligned
&u\left(t,Y_t\right)\\=&u\left(T_{j-1},Y_{j-1}\right)+\int_{T_{j-1}}^t\frac{\partial u}{\partial t}\left(s,Y_s\right)\dif s+\int_{T_{j-1}}^t\nabla u\left(s,Y_s\right)\dif Y_s\\&+\frac{1}{2}\int_{T_{j-1}}^t \Delta u\left(s,Y_s\right)\dif  \langle Y_s,Y_s \rangle\\
=&u\left(T_{j-1},Y_{j-1}\right)-\int_{T_{j-1}}^tb\left(s,Y_s\right)\dif s+\int_{T_{j-1}}^t\<\nabla u\left(s,Y_s\right),\s\left(\eta_{\delta}(s),Y_{\eta_{\delta}(s)}\right)\>\dif W_s\\
&+\int_{T_{j-1}}^t\<\nabla u\left(s,Y_s\right), b\left(\eta_{\delta}(s),Y_{\eta_{\delta}(s)}\right)-b\left(s,Y_s\right)\>\dif s.\\
\endaligned
$$
Hence, we can get
\begin{align}\label{thm-1}
\int_{T_{j-1}}^tb(s,X_s)\dif s=u(T_{j-1},X_{j-1})-u(t,X_t)+\int_{T_{j-1}}^t\<\nabla u(s,X_s),\s(s,X_s)\>\dif W_s,
\end{align}
and
\begin{align}
\int_{T_{j-1}}^tb\left(s,Y_s\right)\dif s=&u\left(T_{j-1},Y_{j-1}\right)-u\left(t,Y_t\right)\no+\int_{T_{j-1}}^t\<\nabla u\left(s,Y_s\right),\s\left(\eta_{\delta}(s),Y_{\eta_{\delta}(s)}\right)\>\dif W_s\no\\
&+\int_{T_{j-1}}^t\<\nabla u\left(s,Y_s\right),b\left(\eta_{\delta}(s),Y_{\eta_{\delta}(s)}\right)-b\left(s,Y_s\right)\>\dif s\label{thm-2}.
\end{align}
Combining with (\ref{thm-1}) and (\ref{thm-2}), we have
$$
\aligned
&X_t-Y_t=X_{T_{j-1}}-Y_{T_{j-1}}\\&+\int_{T_{j-1}}^t\left(b(s,X_s)-b(\eta_{\delta}(s),Y_{\eta_{\delta}(s)})\right)\dif s+\int_{T_{j-1}}^t\left(\s(s,X_s)-\s\left(\eta_{\delta}(s),Y_{\eta_{\delta}(s)}\right)\right)\dif W_s\\
=& X_{T_{j-1}}-Y_{T_{j-1}}+\left(u(T_{j-1},X_{T_{j-1}})-u\left(T_{j-1},Y_{T_{j-1}}\right)\right)-\left(u(t,X_t)-u\left(t,Y_t\right)\right)\\
&+ \int_{T_{j-1}}^t\left[\<\nabla u(s,X_s),\s(s,X_s)\>-\<\nabla u\left(s,Y_s\right),\s\left(\eta_{\delta}(s),Y_{\eta_{\delta}(s)}\right)\>\right]\dif W_s\\
&+\int_{T_{j-1}}^t\<\nabla u\left(s,Y_s\right), b\left(s,Y_s\right)-b\left(\eta_{\delta}(s),Y_{\eta_{\delta}(s)}\right)\>\dif s+ \int_{T_{j-1}}^t\left(b\left(s,Y_s\right)-b\left(\eta_{\delta}(s),Y_{\eta_{\delta}(s)}\right)\right)\dif s\\&+\int_{T_{j-1}}^t\left(\s(s,X_s)-\s\left(\eta_{\delta}(s),Y_{\eta_{\delta}(s)}\right)\right)\dif W_s.\\
\endaligned
$$
By (\ref{a}) and the mean-value theorem, we have:
$$
\aligned
&\left|X_t-Y_t\right| \\ \leq& \left|X_{T_{j-1}}-Y_{T_{j-1}}\right|+\left|u(T_{j-1},X_{T_{j-1}})-u\left(T_{j-1},Y_{T_{j-1}}\right)\right|+\left|u
(t,X_t)-u\left(t,Y_t\right)\right|\\
&\quad+ \left|\int_{T_{j-1}}^t\left[\<\nabla u(s,X_s),\s(s,X_s)\>-\<\nabla u\left(s,Y_s\right),\s\left(\eta_{\delta}(s),Y_{\eta_{\delta}(s)}\right)\>\right]\dif W_s\right|\\
&\quad+\|\nabla u\|_{C_b^{\beta}[T_{j-1},T_j]}\int_{T_{j-1}}^t\left|b\left(s,Y_s\right)-b\left(\eta_{\delta}(s),Y_{\eta_{\delta}(s)}\right)\right|\dif s\\
&\quad+ \int_{T_{j-1}}^t\left|b\left(s,Y_s\right)-b\left(\eta_{\delta}(s),Y_{\eta_{\delta}(s)}\right)\right|\dif s+\left|\int_{T_{j-1}}^t\left(\s(s,X_s)-\s\left(\eta_{\delta}(s),Y_{\eta_{\delta}(s)}\right)\right)\dif W_s\right|\\
&\leq (1+\varepsilon)\left|X_{T_{j-1}}-Y_{T_{j-1}}\right|+\varepsilon\left|X_t-Y_t\right|\\
&\quad+ \left|\int_{T_{j-1}}^t\left[\<\nabla u(s,X_s),\s(s,X_s)\>-\<\nabla u\left(s,Y_s\right),\s\left(\eta_{\delta}(s),Y_{\eta_{\delta}(s)}\right)\>\right]\dif W_s\right|\\
&\quad+ (1+\varepsilon)\int_{T_{j-1}}^t\left|Y_s-Y_{\eta_{\delta}(s)}\right|^{\b}\phi\left(\left|Y_s-Y_{\eta_{\delta}(s)}\right|\right)\dif s+(1+\varepsilon)\int_{T_{j-1}}^t\phi\left(\left|s-\eta_{\delta}(s)\right|\right)\dif s\\
&\quad+\left|\int_{T_{j-1}}^t\left(\s(s,X_s)-\s\left(\eta_{\delta}(s),Y_{\eta_{\delta}(s)}\right)\right)\dif W_s\right|.
\endaligned
$$
For all $p\geq 1$,  utilizing Jensen's inequality, H\"older inequality and Lemma \ref{lem1}, we can obtain
$$
\aligned
&\left|X_t-Y_t\right|^p\\
\leq& 6^{p-1}(1+\varepsilon)^p\left|X_{T_{j-1}}-Y_{T_{j-1}}\right|^p+6^{p-1}\varepsilon^p\left|X_t-Y_t\right|^p\\
& +6^{p-1}\left|\int_{T_{j-1}}^t\left[\<\nabla u(s,X_s),\s(s,X_s)\>-\<\nabla u\left(s,Y_s\right),\s\left(\eta_{\delta}(s),Y_{\eta_{\delta}(s)}\right)\>\right]\dif W_s\right|^p\\
&+6^{p-1}(1+\varepsilon)^p(t-T_{j-1})^{p-1}M^p\int_{T_{j-1}}^t\left|Y_s-Y_{\eta_{\delta}(s)}\right|^{p(\b+1)}\dif s\\
& +6^{p-1}(1+\varepsilon)^p(t-T_{j-1})^{p-1}M^p\int_{T_{j-1}}^t\left|s-\eta_{\delta}(s)\right|^p\dif s\\
& +6^{p-1}\left|\int_{T_{j-1}}^t\left(\s(s,X_s)-\s\left(\eta_{\delta}(s),Y_{\eta_{\delta}(s)}\right)\right)\dif W_s\right|^p.\\
\endaligned
$$
Because $\varepsilon$ is arbitrary, there exists $c(p,\varepsilon):=6^{p-1}\varepsilon^p<1$. Then we know
\begin{align}
&\left|X_t-Y_t\right|^p\leq \frac{6^{p-1}(1+\varepsilon)^p}{1-c(p,\varepsilon)}\left|X_{T_{j-1}}-Y_{T_{j-1}}\right|^p\no\\
\quad &+\frac{6^{p-1}}{1-c(p,\varepsilon)}\left|\int_{T_{j-1}}^t\left[\<\nabla u(s,X_s),\s(s,X_s)\>-\<\nabla u\left(s,Y_s\right),\s\left(\eta_{\delta}(s),Y_{\eta_{\delta}(s)}\right)\>\right]\dif W_s\right|^p\no\\
\quad &+\frac{6^{p-1}(1+\varepsilon)^p(t-T_{j-1})^{p-1}M^p}{1-c(p,\varepsilon)}\int_{T_{j-1}}^t\left[\left|Y_s-Y_{\eta_{\delta}(s)}\right|^{p(\b+1)}+(t-T_{j-1})\delta^p\right]\dif s\no\\
\quad &+\frac{6^{p-1}}{1-c(p,\varepsilon)}\left|\int_{T_{j-1}}^t\left(\s(s,X_s)-\s\left(\eta_{\delta}(s),Y_{\eta_{\delta}(s)}\right)\right)\dif W_s\right|^p\label{thm-3}.
\end{align}
Taking the supremum, expectation on both sides of  the above inequality, and using BDG's inequality, for $t\in(T_{j-1},T_j]$, we have
\begin{align}
&\mE\left[\sup_{T_{j-1}\leq u\leq t}\left|X_u-Y_u\right|^p\right]\leq \frac{6^{p-1}(1+\varepsilon)^p}{1-c(p,\varepsilon)}\mE\left[\left|X_{T_{j-1}}-Y_{T_{j-1}}\right|^p\right]+\frac{6^{p-1}C(p,d)T^{\frac{p}{2}-1}}{1-c(p,\varepsilon)}\no\\
&\quad \times\int_{T_{j-1}}^t\mE\left[\sup_{T_{j-1}\leq u\leq s}\left\|\nabla u(u,X_u)\s(u,X_u)-\nabla u\left(u,Y_u\right)\s\left(\eta_{\delta}(u),Y_{\eta_{\delta}(u)}\right)\right\|_{\HS}^p\right]\dif s\no\\
&\quad +\frac{6^{p-1}(1+\varepsilon)^pT^{p-1}M^p}{1-c(p,\varepsilon)}\left[\int_{T_{j-1}}^t\mE\left[\sup_{T_{j-1}\leq u\leq s}\left|Y_u-Y_{\eta_{\delta}(u)}
\right|^{p(\b+1)}\right]\dif s+T\delta^p\right]\no\\
&\quad +\frac{6^{p-1}C(p,d)T^{\frac{p}{2}-1}}{1-c(p,\varepsilon)}\int_{T_{j-1}}^t\mE\left[\sup_{T_{j-1}\leq u\leq s}\left\|\s(u,X_u)-\s\left(\eta_
{\delta}(u),Y_{\eta_{\delta}(u)}\right)\right\|_{\HS}^p\right]\dif s\no\\
&:=\sum_{i=1}^4\mI_i,\no
\end{align}
where $C(p,d)$ is the constant in BDG's inequality. With the help of lemma \ref{lem1} and the Assumption \ref{assum-1} $(a)$, in $\mI_2$,  we have
\begin{align}
&\mE\left[\sup_{T_{j-1}\leq u\leq s}\left\|\nabla u(u,X_u)\s(u,X_u)-\nabla u\left(u,Y_u\right)\s\left(\eta_{\delta}(u),Y_{\eta_{\delta}(u)}\right)\right\|_{\HS}^p\right]\no\\
=&\mE\left[\sup_{T_{j-1}\leq u\leq s}\left\|\nabla u(u,X_u)\s(u,X_u)-\nabla u\left(u,Y_u\right)\s(u,X_u)+\nabla u\left(u,Y_u\right)\s(u,X_u)\right.\right.\no\\
&\left.-\nabla u\left(u,Y_u\right)\s\left(u,Y_u\right)+\nabla u\left(u,Y_u\right)\s\left(u,Y_u\right)-\nabla u\left(u,Y_u\right)\s\left(\eta_{\delta}(u),Y_u\right)\right.\no\\ &\left.
+\nabla u\left(u,Y_u\right)\s\left(\eta_{\delta}(u),Y_u\right)
\left.\left.-\nabla u\left(u,Y_u\right)\s\left(\eta_{\delta}(u),Y_{\eta_{\delta}(u)}\right)\notag\right.\right\|_{\HS}^p\right]\no\\
\leq& 4^{p-1}2^p\varepsilon^p\|\s\|_{T,\infty}^p+4^{p-1}\varepsilon^p\mE\left[\sup_{T_{j-1}\leq u\leq s}\phi(|X_u-Y_u|)^p\right]\no\\
&+4^{p-1}\varepsilon^p\mE\left[\sup_{T_{j-1}\leq u\leq s}\phi(|u-\eta_{\delta}(u)|)^p\right]
+4^{p-1}\varepsilon^p\mE\left[\sup_{T_{j-1}\leq u\leq s}\phi(|Y_u-Y_{\eta_{\delta}(u)}|)^p\right]\no\\
\leq & 4^{p-1}2^p\varepsilon^p\|\s\|_{T,\infty}^p+4^{p-1}\varepsilon^pM^p\mE\left[\sup_{T_{j-1}\leq u\leq s}|X_u-Y_u|^p\right]+4^{p-1}\varepsilon^pM^p\delta^p+4^{p-1}\varepsilon^pC^p\delta^{p/2}\no.
\end{align}
In $\mI_4$, from the properties of Dini-function, it may be chosen the constant $C_0,C_1$, such that
$$
\begin{aligned}
&\mE\left[\sup_{T_{j-1}\leq u\leq s}\left\|\s(u,X_u)-\s\left(\eta_{\delta}(u),Y_{\eta_{\delta}(u)}\right)\right\|_{\HS}^p\right]\\
=&\mE\left[\sup_{T_{j-1}\leq u\leq s}\left\|\s(u,X_u)-\s\left(\eta_{\delta}(u),X_u\right)+\s\left(\eta_{\delta}(u),X_u\right)-\s\left(\eta_{\delta}(u),Y_u\right)+
\s\left(\eta_{\delta}(u),Y_u\right)\right.\right.\\
&-\left.\left.\s\left(\eta_{\delta}(u),Y_{\eta_{\delta}(u)}\right)\right\|_{\HS}^p\right]\\
\leq&3^{p-1}\mE\left[\sup_{T_{j-1}\leq u\leq s}\phi(|u-\eta_{\delta}(u)|)^p\right]+3^{p-1}\mE\left[\sup_{T_{j-1}\leq u\leq s}\phi(|X_u-Y_u|)^p\right]\\&+3^{p-1}\mE\left[\sup_{T_{j-1}\leq u\leq s}\phi(|Y_u-Y_{\eta_{\delta}(u)}|)^p\right] \\
\leq&3^{p-1}M^p\delta^p+3^{p-1}M^p\mE\left(\sup_{T_{j-1}\leq u\leq s}|X_u-Y_u|^p\right)+3^{p-1}M^pC^p\delta^{p/2}.\\
\end{aligned}
$$
Thus, for $\delta\in(0,1)$ and $p\geq1$, there exists the constant $C_2,C_3, C_4$, we know
$$
\aligned
&\mE\left[\sup_{T_{j-1}\leq u\leq t}\left|X_u-Y_u\right|^p\right]\leq \frac{6^{p-1}(1+\varepsilon)^p}{1-c(p,\varepsilon)}\mE\left[\left|X_{T_{j-1}}-Y_{T_{j-1}}\right|\right]^p\\
&+\frac{6^{p-1}c(p,d)T^{\frac{p}{2}-1}(\varepsilon^pM^p4^{p-1}+3^{p-1}M^p)}{1-c(p,\varepsilon)}\int_{T_{j-1}}^t\mE\left[\sup_{T_{j-1}\leq u\leq s}\left|X_u-Y_{u}\right|^p\right]\dif s\\
&+\frac{6^{p-1}(1+\varepsilon)^pT^{\frac{p}{2}-1}}{1-c(p,\varepsilon)}\\
&\ \  \ \ \ \
\times\left[4^{p-1}\varepsilon^pC^p\delta^{p/2}+4^{p-1}\varepsilon^pM^p\delta^p+4^{p-1}2^p\varepsilon^pK^p +3^{p-1}
C^p\delta^{p/2}+3^{p-1}M^p\delta^p\right]\\
&+\frac{6^{p-1}T^{p}(1+\varepsilon)^pM^p}{1-c(p,d,\varepsilon)}\left[C\delta^{p(\b+1)/2}+T\delta^p\right]\\
\leq&C_2\mE\left[\left|X_{T_{j-1}}-Y_{T_{j-1}}\right|\right]^p+C_3\int_{T_{j-1}}^t\mE\left[\sup_{T_{j-1}\leq u\leq s}\left|X_u-Y_{u}\right|^p\right]\dif s+C_4\delta^{p/2}.\\
\endaligned
$$
Next, we prove by the Lemma \ref{lem1} that for each $j=1,\cdots,m$,
\begin{equation}\label{3}
\mE\left[\sup_{T_{j-1}\leq u\leq t}\left|X_u-Y_u\right|^p\right]\leq A_j\delta^{p/2},\ \ t\in(T_{j-1},T_j].
\end{equation}
where $A_1=C_4e^{C_3T}$ and $A_j=(C_2A_{j-1}+C_4)e^{C_3T}$, for $j=2,\cdots,m$. If  $j=1$, since $T_0=0$,  $\forall t\in(0,T_1]$, we have
$$
\mE\left[\sup_{0\leq u\leq t}\left|X_u-Y_u\right|^p\right]\leq C_3\int_0^t\mE \left[\sup_{0\leq u\leq s}\left|X_u-Y_u\right|^p\right]\dif s+C_4\delta^{p/2}.
$$
Using Gronwall's  inequality, we can get
$$
\mE\left[\sup_{0\leq u\leq t}\left|X_u-Y_u\right|^p\right]\leq C_4e^{C_3T}\delta^{p/2},\ \ t\in(0,T_1].
$$
We assume that   (\ref{3})  holds  for  $j=1,2,\cdots,i-1$ with $2\leq i\leq m$. Then $\forall t\in(T_{i-1},T_{i}]$, we realize
$$
\aligned
\mE\left[\sup_{T_{i-1}\leq u\leq t}\left|X_u-Y_u\right|^p\right]\leq &C_2\mE\left[\left|X_{T_{i-1}}-Y_{T_{i-1}}\right|^p\right]\\
&+C_3\int_{T_{i-1}}^{t}\mE\left[\sup_{T_{i-1}\leq u\leq s}\left|X_u-Y_u\right|^p\right]\dif s+C_4\delta^{p/2}\\
\leq &C_3\int_{T_{i-1}}^{t}\mE\left[\sup_{T_{i-1}\leq u\leq s}\left|X_u-Y_u\right|^p\right]\dif s+(C_2A_{i-1}+C_4)\delta^{p/2}.\\
\endaligned
$$
By once more Gronwall's inequality, it holds that
$$
\mE\left[\sup_{T_{i-1}\leq u\leq t}\left|X_u-Y_u\right|^p\right]\leq (C_2A_{i-1}+C_4)e^{C_2T}\delta^{p/2}=A_i\delta^{p/2}, \ \ t\in(T_{i-1},T_{i}].
$$
Hence $\forall j=1,\cdots,m$, (\ref{3}) is true.  And we draw the conclusion that
$$
\mE\left[\sup_{0\leq s\leq T}\left|X_u-Y_u\right|^p\right]\leq \sum_{j=1}^m\mE\left[\sup_{T_{j-1}\leq u\leq T_j}\left|X_u-Y_u\right|^p\right]
\leq \sum_{j=1}^mA_j\delta^{p/2}:=\beta\delta^{p/2}.
$$
So the proof is finished.
\end{proof}

\begin{proof}[Proof of Theorem \ref{thm2}]
Let $\chi\in C_b^{\infty}(\mR_+)$ is the cut-off function,  such that $0\leq \chi\leq 1$,  $\chi(r)=1$ for $r\in(0,1)$, and $\chi(r)=0$,  for $r\geq 2$ .  For any $t\in[0,T]$ and $k\geq 1$, let
$$
b^{(k)}(t,x)=b(t,x)\chi\left(\frac{|x|}{k}\right),\ \ \sigma^{(k)}(t,x)=\sigma\left(t,\chi\left(\frac{|x|}{k}x\right)\right),\ \ x\in\mR^n.
$$
Fixed $k\geq 1$, we have,
$$
X_t^{(k)}=x+\int_0^tb^{(k)}\left(s,X_s^{(k)}\right)\dif s+\int_0^t\s^{(k)}\left(s,X_s^{(k)}\right)\dif W_s,\ \ t\in(0,T].
$$
The corresponding continuous-time Euler-Maruyama is
\begin{equation}\label{4}
Y_t^{(k)}=x+\int_0^tb^{(k)}\left(\eta_{\delta}(s),Y_{\eta_{\delta}(s)}^{(k)}\right)\dif s+\int_0^t\s^{(k)}\left(\eta_{\delta}(s),Y_{\eta_{\delta}(s)}^{(k)}\right)\dif W_s, \ \ t\in(0,T].
\end{equation}
Using the BDG, H\"older and Gronwall inequality, for all $p\geq1$, for some $C_T$, we have (see the proof Theorem 1.2 in \cite{BHY})
\begin{equation}\label{5}
\aligned
&\mE\left(\sup_{0\leq t\leq T}|X_t|^p\right)+\mE\left(\sup_{0\leq t\leq T}\left|Y_t\right|^p\right)+\mE\left(\sup_{0\leq t\leq T}\left|X_t^{(k)}\right|^p\right)+\mE\left(\sup_{0\leq t\leq T}\left|Y_t^{(k)}\right|^p\right)\\ \leq &C_T\left(1+\mE|X_0|^p\right)<+\infty.\\
\endaligned
\end{equation}
Since
$$
\aligned
\mE\left(\sup_{0\leq t\leq T}\left|X_t-Y_t\right|^p\right)\leq&3^{p-1}\mE\left(\sup_{0\leq t\leq T}\left|X_t-X_t^{(k)}\right|^p\right) \\
&+ 3^{p-1}\mE\left(\sup_{0\leq t\leq T}\left|X_t^{(k)}-Y_t^{(k)}\right|^p\right)\\
&+ 3^{p-1}\mE\left(\sup_{0\leq t\leq T}\left|Y_t-Y_t^{(k)}\right|^p\right)\\
=: &\mI_1+\mI_2+\mI_3,\\
\endaligned
$$
applying the Chebyshev inequality and $(\ref{5})$, we can deduce
$$
\aligned
\mI_1+\mI_3\preceq&\mE\left(\sup_{0\leq t\leq T}\left|X_t-X_t^{(k)}\right|^p\mI_{\{\sup_{0\leq t\leq T}|X_t|\geq k\}}\right)\\
&+\mE\left(\sup_{0\leq t\leq T}\left|Y_t-Y_t^{(k)}\right|^p\mI_{\{\sup_{0\leq t\leq T}\left|Y_t\right|\geq k\}}\right)\\
\preceq& \left(\mE\left(\sup_{0\leq t\leq T}|X_t|^p\right)+\mE\left(\sup_{0\leq t\leq T}\left|X_t^{(k)}\right|^p\right)\right)\frac{\mE\left(\sup_{0\leq t\leq T}|X_t|\right)}{k}\\
&+\left(\mE\left(\sup_{0\leq t\leq T}\left|Y_t\right|^p\right)+\mE\left(\sup_{0\leq t\leq T}\left|Y_t^{(k)}\right|^p\right)\right)\frac{\mE\left(\sup_{0\leq t\leq T}\left|Y_t\right|\right)}{k}\\
\preceq&\frac{1}{k}.\\
\endaligned
$$
For the terms $\mI_2$, by the Theorem \ref{thm1}, we have
$$
\mI_2\preceq\beta_1\delta^{p/2},
$$
where the constant $\beta_1>1$ depending on $T, p, d, \|\s\|_{T,\infty}, M, k, \b$. Consequently, we conclude that
$$
\mE\left(\sup_{0\leq t\leq T}\left|X_t-Y_t\right|^p\right)\preceq\frac{1}{k}+\b_1\d^{\frac{p}{2}}.
$$
For any $\varepsilon>0$, taking $k=\frac{1}{\varepsilon}$ and $\delta\rightarrow 0$, implies that
$$
\lim_{\delta\rightarrow 0}\mE\left(\sup_{0\leq t\leq T}\left|X_t-Y_t\right|^p\right)=0.
$$
Thus, the proof of Theorem \ref{thm2} can be complete.
\end{proof}

\end{document}